\documentclass[reqno]{amsart}

\usepackage{amssymb,amscd}

\theoremstyle{plain}

\newtheorem{thm}{Theorem}[section]
\newtheorem{lem}[thm]{Lemma}
\newtheorem{cor}[thm]{Corollary}

\theoremstyle{definition}

\theoremstyle{remark}

\newtheorem{claim}{Claim}

\newcommand{\Z}{\mathbb{Z}}
\newcommand{\R}{\mathbb{R}}

\newcommand{\N}{\mathbb{N}}

\newcommand{\KK}{\mathcal{K}}
\newcommand{\LL}{\mathcal{L}}

\renewcommand{\SS}{\mathcal{S}}

\newcommand{\length}{\operatorname{length}}

\newcommand{\Cinf}{C^\infty}

\title[Generalized Hermite polynomials]{Application of the method of Bonan-Clark to the generalized Hermite polynomials}

\author[J.A. \'Alvarez L\'opez]{Jes\'us A. \'Alvarez L\'opez}
\address{Departamento de Matem\'aticas\\
         Facultade de Matem\'aticas\\
         Universidade de Santiago de Compostela\\
         15782 Santiago de Compostela\\ Spain}
\email{jesus.alvarez@usc.es}
\thanks{The first author is partially supported by MICINN, Grant MTM2014-56950-P, and by Xunta de Galicia, Consolidaci\'on e estructuraci\'on 2015 GPC GI-1574.}

\author[M. Calaza]{Manuel Calaza}
\address{Laboratorio de Investigaci\'on 10, IDIS, Fundaci\'on IDICHUS, Hospital Cl\'inico Universitario, 15706 Santiago de Compostela, Spain}
\email{manuel.calaza@usc.es}

\date{}



\begin{document}

\maketitle

\begin{abstract}
  It is shown that the method of Bonan-Clark can be used to prove asymptotic estimates for the generalized Hermite polynomials, showing also that these estimates are optimal.
\end{abstract}


\section{Introduction} \label{s:intro}

Let $p_k$ be the sequence of orthogonal polynomials for the measure $e^{-sx^2}|x|^{2\sigma}\,dx$, taken with norm one and positive leading coefficient. Up to normalization, these are the generalized Hermite polynomials~\cite[p.~380, Problem~25]{Szego1975}; see also \cite{Chihara1955,DickinsonWarsi1963,DuttaChatterjeaMore1975,Chihara1978,Rosenblum1994,Rosler1998}. Let $x_{k,k}<x_{k,k-1}<\dots<x_{k,1}$ denote the roots of each $p_k$; in particular, $x_{k,k/2}$ is the smallest positive root if $k$ is even. The generalized Hermite functions $\phi_k=p_ke^{-sx^2/2}$ are the eigenfunctions of the Dunkl harmonic oscillator on $\R$ \cite{Rosenblum1994}. 

Asymptotic estimates for the functions $\phi_k$, or the polynomials $p_k$, can be obtained by using that the generalized Hermite polynomials can be expressed by the Laguerre ones (see e.g.\ \cite[p.~525]{Rosler1998} or \cite[p.~23]{Rosler2003}), which have asymptotic estimates \cite{Erdelyi1960,AskeyWainger1965,Muckenhoupt1970a-II,Muckenhoupt1970b}. This procedure is indicated in Section~\ref{ss:Laguerre}. 

We show that method of Bonan-Clark, used in \cite{BonanClark1990} for the Hermite functions, can be also applied to the functions $\xi_k=|x|^\sigma\phi_k$. They satisfy the equation $\xi_k''+q_k\xi_k=0$, where $q_k=(2k+1+2\sigma)s-s^2x^2-\bar\sigma_kx^{-2}$ with $\bar\sigma_k=\sigma(\sigma-(-1)^k)$. The corresponding oscillation region $\widehat{I}_k=q^{-1}(\R_+)$ is of the form: $(-b_k,-a_k)\cup(a_k,b_k)$ if $\bar\sigma_k>0$ (for $k>0$), $(-b_k,b_k)$ if  $\bar\sigma_k=0$, or $(-b_k,0)\cup(0,b_k)$ if  $\bar\sigma_k<0$, where $b_k\in O(k^{1/2})$ and $a_k\in O(k^{-1/2})$ as $k\to\infty$. If $\bar\sigma_k\ge0$, then set $\widehat{J}_k=\widehat{I}_k$. When $\bar\sigma_k<0$ and $k$ is large enough, the equation $q_k(b)=4\pi/b^2$ has two positive solutions, $b_{k,+}<b_{k,-}$, with $b_{k,+}\in O(k^{-1/2})$; in this case, set $\widehat{J}_k=(-b_k,-b_{k,+}]\cup[b_{k,+},b_k)$.
  
  \begin{thm}\label{t:upper estimates of xi k}
  There exist $C,C',C''>0$, depending on $\sigma$ and $s$, so that, for $k\ge1$:
    \begin{itemize}

      \item[(i)] $\xi_k^2(x)\le C/\sqrt{q_k(x)}$ for all $x\in\widehat{J}_k$\;;

      \item[(ii)] if $k$ is odd or $\sigma\ge0$, then $\xi_k^2(x)\le C'k^{-1/6}$ for all $x\in\R$\;; and

      \item[(iii)] if $k$ is even and $\sigma<0$, then $\xi_k^2(x)\le C''k^{-1/6}$ for $|x|\ge x_{k,k/2}$.

    \end{itemize}
\end{thm}

In the case of Theorem~\ref{t:upper estimates of xi k}-(iii), the estimate of $\xi_k$ cannot be extended to $\R\setminus\{0\}$ because these functions are unbounded near zero. Therefore some condition of the type $|x|\ge x_{k,k/2}$ must be assumed; the meaning of this condition is clarified by pointing out that $x_{k,k/2}\in O(k^{-1/2})$ as $k\to\infty$. This weakness is complemented by the following result.

\begin{thm}\label{t:upper estimates of phi k for k even and sigma<0}
  Suppose that $\sigma<0$. There exist $C'''>0$, depending on $\sigma$ and $s$, such that $\phi_k^2(x)\le C'''$ for all $k$ and $|x|\le1$.
\end{thm}

Even though Theorems~\ref{t:upper estimates of xi k} and~\ref{t:upper estimates of phi k for k even and sigma<0} easily follow from the known asymptotic estimates of Laguerre polynomials, we think that the given adaptation of the method of Bonan-Clark has its own interest. Moreover this method continues with the proof of the following theorem asserting that the estimates of Theorem~\ref{t:upper estimates of xi k}-(ii),(iii) are optimal.

\begin{thm}\label{t:lower estimates of max xi k^2}
  There are $C^{(4)},C^{(5)}>0$, depending on $\sigma$ and $s$, so that, for $k\ge1$:
    \begin{itemize}

      \item[(i)] $\max_{x\in\R}\xi_k^2(x)\ge C^{(4)}k^{-1/6}$; and,

      \item[(ii)] if $k$ is even and $\sigma<0$, then $\max_{|x|\ge x_{k,k/2}}\xi_k^2(x)\ge C^{(5)}k^{-1/6}$.

    \end{itemize}
\end{thm}

The method that Bonan-Clark has two steps: first, it estimates the distance from any point $x$ in the oscillation region $\widehat{I}_k$ to some root $x_{k,i}$, and, second, the value of $\xi_k^2(x)$ is estimated by using $|x-x_{k,i}|$. These computations become much more involved than in \cite{BonanClark1990}; indeed, several cases are considered separately, some of them with significant differences; for instance, some roots $x_{k,i}$ may not be in the oscillation region $\widehat{I}_k$, and the functions $\xi_k$ may not be bounded, as we said. 

The asymptotic distribution of the roots $x_{k,i}$ as $k\to\infty$ also has a well known measure theoretic interpretation \cite{ErdosTuran1940,VanAssche1985,VanAsscheTeugels1987}; specially, the generalized Hermite polynomials are considered in \cite[Section~4]{VanAssche1985}. However the weak convergence of measures considered in those publications does not seem to give the asymptotic approximation of the roots needed in the first step.

\section{Preliminaries}\label{s:preliminaries}

\subsection{Dunkl operator}\label{ss:D sigma}

Recall that, for any $\phi\in\Cinf=\Cinf(\R)$, there is some $\psi\in\Cinf$ such that $\phi(x)-\phi(0)=x\psi(x)$, which also satisfies
  \begin{equation}\label{psi (m)(x)}
    \psi^{(m)}(x)=\int_0^1t^m\phi^{(m+1)}(tx)\,dt
  \end{equation}
for all $m\in\N$ (see e.g.\ \cite[Theorem~1.1.9]{Hormander1990}). The notation $\psi=x^{-1}\phi$ is used.

The Dunkl operator, in the case of dimension one, is the differential-difference operator $T_\sigma$ on $\Cinf$, depending on a parameter $\sigma\in\R$, defined by
  \[
    (T_\sigma\phi)(x)=\phi'(x)+\sigma\,\frac{\phi(x)-\phi(-x)}{x}\;.
  \]
It can be considered as a perturbation of the derivative operator $\frac{d}{dx}$.

Consider the decomposition $\Cinf=\Cinf_{\text{\rm ev}}\oplus\Cinf_{\text{\rm odd}}$, as direct sum of subspaces of even and odd functions. The matrix expressions of operators on $\Cinf$ will be considered with respect to this decomposition. The operator of multiplication by a function $h$ will be denoted also by $h$. We can write
  $
    \frac{d}{dx}=
      \begin{pmatrix}
        0 &  \frac{d}{dx} \\
        \frac{d}{dx} & 0
      \end{pmatrix}
  $,
  $
    x=
      \begin{pmatrix}
        0 &  x \\
        x & 0
      \end{pmatrix}
  $ and
  $$
    T_\sigma=
      \begin{pmatrix}
        0 &  \frac{d}{dx}+2\sigma x^{-1} \\
        \frac{d}{dx} & 0
      \end{pmatrix}
    = \frac{d}{dx}+2\sigma
      \begin{pmatrix}
        0 &  x^{-1} \\
        0 & 0
      \end{pmatrix}
  $$
on $\Cinf$. With
$
  \Sigma=
    \begin{pmatrix}
      \sigma & 0 \\
      0 & -\sigma
    \end{pmatrix}
$, we have
  \begin{gather}
    [T_\sigma,x]=1+2\Sigma\;,\label{[D sigma,x]=1+2Sigma}\\
    T_\sigma\Sigma+\Sigma T_\sigma=x\,\Sigma+\Sigma\,x=0\;.
    \label{D sigma Sigma+Sigma D sigma=x Sigma+Sigma x=0}
  \end{gather}
Let $m!_\sigma$ denote the perturbed factorial of each $m\in\N$, which is inductively defined by setting $0!_\sigma=1$, and
  \[
    m!_\sigma=
      \begin{cases}
        (m-1)!_\sigma m & \text{if $m$ is even}\\
        (m-1)!_\sigma(m+2\sigma) & \text{if $m$ is odd}
      \end{cases}
  \]
for $m>0$. Observe that $m!_\sigma>0$ if $\sigma>-1/2$. For $k\le m$, even when $k!_\sigma=0$, the quotient $m!_\sigma/k!_\sigma$ can be understood as the product of the factors from the definition of $m!_\sigma$ which are not included in the definition of $k!_\sigma$. For any $\phi\in\Cinf$ and $m\in\N$, we have
  \begin{equation}\label{(T sigma m phi)(0)}
      (T_\sigma^m\phi)(0)=\frac{m!_\sigma}{m!}\phi^{(m)}(0)\;.
  \end{equation}
This equality follows by~\eqref{psi (m)(x)} and induction on $m$.

\subsection{Dunkl harmonic oscillator}\label{ss:L}

Recall that, for dimension one, the harmonic oscillator, and the annihilation and creation operators are
$H=-\frac{d^2}{dx^2}+s^2x^2$, $A=sx+\frac{d}{dx}$ and $A'=sx-\frac{d}{dx}$ on $\Cinf$. By using $T_\sigma$ instead of $d/dx$, we get the Dunkl harmonic oscillator, and Dunkl annihilation and creation operators:
  \begin{gather*}
    L=-T_\sigma^2+s^2x^2=H-2\sigma
      \begin{pmatrix}
        x^{-1}\,\frac{d}{dx} & 0 \\
        0 & \frac{d}{dx}\,x^{-1}
      \end{pmatrix}\;,\\
    B=sx+T_\sigma=A+2\sigma
      \begin{pmatrix}
        0 &  x^{-1} \\
        0 & 0
      \end{pmatrix}\;,\\
      B'=sx-T_\sigma=A'-2\sigma
      \begin{pmatrix}
        0 &  x^{-1} \\
        0 & 0
      \end{pmatrix}\;.
  \end{gather*}
By~\eqref{[D sigma,x]=1+2Sigma} and~\eqref{D sigma Sigma+Sigma D sigma=x Sigma+Sigma x=0},
  \begin{gather}
    L=BB'-(1+2\Sigma)s=B'B+(1+2\Sigma)s=\frac{1}{2}(BB'+B'B)\;,\label{L}\\
    [L,B]=-2sB\;,\quad[L,B']=2sB'\;,\label{[L,B]}\\
    [B,B']=2s(1+2\Sigma)\;,\label{[B,B']}\\
    [L,\Sigma]=B\Sigma+\Sigma B=B'\Sigma+\Sigma B'=0\;.
    \label{[L,Sigma]=B' Sigma+Sigma B'=0}
  \end{gather}

Recall also that the Schwartz space $\SS=\SS(\R)$ consists of the functions $\phi\in\Cinf$ such that $\|\phi\|_{\SS^m}=\sum_{i+j\le m}\sup_x|x^i\phi^{(j)}(x)|$ is finite for all $m\in\N$ (including zero\footnote{We adopt the convention $0\in\N$.}). This defines a sequence of norms $\|\ \|_{\SS^m}$ on $\SS$, which is endowed with the corresponding Fr\'echet topology. The Banach space completion of $\SS$ with respect to each norm $\|\ \|_{\SS^m}$ will be denoted by $\SS^m$. We have $\SS^{m+1}\subset\SS^m$ continuously\footnote{Let $X$ and $Y$ be topological vector spaces. It is said that $X\subset Y$ continuously if $X$ is a linear subspace of $Y$ and the inclusion map $X\hookrightarrow Y$ is continuous.}, and $\SS=\bigcap_m\SS^m$. Let us remark that $\|\phi'\|_{\SS^m}\le\|\phi\|_{\SS^{m+1}}$ for all $m$.

The above decomposition of $\Cinf$ can be restricted to each $\SS^m$ and $\SS$,
giving $\SS^m=\SS^m_{\text{\rm ev}}\oplus\SS^m_{\text{\rm odd}}$ and $\SS=\SS_{\text{\rm ev}}\oplus\SS_{\text{\rm odd}}$. The
matrix expressions of operators on $\SS$ will be considered with
respect to this decomposition. For $\phi\in \Cinf_{\text{\rm ev}}$, $\psi=x^{-1}\phi$ and $i,j\in\N$, it follows from~\eqref{psi (m)(x)} that
  \[
    |x^i\psi^{(j)}(x)|\le\int_0^1t^{j-i}|(tx)^i\phi^{(j+1)}(tx)|\,dt\le\sup_{y\in\R}|y^i\phi^{(j+1)}(y)|
  \]
for all $x\in\R$. Thus $\|\psi\|_{\SS^m}\le\|\phi\|_{\SS^{m+1}}$ for all $m\in\N$, obtaining that $\SS_{\text{\rm odd}}=x\,\SS_{\text{\rm ev}}$ and $x^{-1}:\Cinf_{\text{\rm odd}}\to \Cinf_{\text{\rm ev}}$ restricts to a continuous operator $x^{-1}:\SS_{\text{\rm odd}}\to\SS_{\text{\rm ev}}$. Therefore $x:\SS_{\text{\rm ev}}\to\SS_{\text{\rm odd}}$ is an isomorphism of Fr\'echet spaces, and $T_\sigma$, $B$, $B'$ and $L$ define continuous operators on $\SS$.

Let $\langle\ ,\ \rangle_\sigma$ and $\|\ \|_\sigma$ denote the scalar product and norm of $L^2(\R,|x|^{2\sigma}\,dx)$. Assume from now on that $\sigma>-1/2$, and therefore $\SS$ is dense in $L^2(\R,|x|^{2\sigma}\,dx)$. In $L^2(\R,|x|^{2\sigma}\,dx)$, with domain $\SS$, $-T_\sigma$ is adjoint of $T_\sigma$, $B'$ is adjoint of $B$, and $L$ is essentially self-adjoint. The self-adjoint extension of $L$ (with domain $\SS$) will be denoted by $\LL$, or $\LL_\sigma$. Its spectrum consists of the eigenvalues $(2k+1+2\sigma)s$ ($k\in\N$), with normalized eigenfunctions $\phi_k$ inductively defined by
  \begin{align}
    \phi_0&=s^{(2\sigma+1)/4}\Gamma(\sigma+1/2)^{-1/2}e^{-sx^2/2}\;,\label{phi 0}\\
    \phi_k&=
      \begin{cases}
        (2ks)^{-1/2}B'\phi_{k-1} & \text{if $k$ is even}\\
        (2(k+2\sigma)s)^{-1/2}B'\phi_{k-1} & \text{if $k$ is odd}
      \end{cases}
    \label{phi k}
  \end{align}
for $k\ge1$. We also have $B\phi_0=0$ and
    $$
      B\phi_k=
        \begin{cases}
          (2ks)^{1/2}\phi_{k-1} & \text{if $k$ is even}\\
          (2(k+2\sigma)s)^{1/2}\phi_{k-1} & \text{if $k$ is odd}
        \end{cases}
    $$
  for $k\ge1$. These assertions follow from \eqref{L}--\eqref{[L,Sigma]=B' Sigma+Sigma B'=0} like in the case of $H$.

\subsection{Generalized Hermite polynomials}\label{ss:p k}

From~\eqref{phi 0},~\eqref{phi k} and the definition of $B'$, it follows that $\phi_k=p_ke^{-sx^2/2}$, where $p_k$ is the sequence of polynomials inductively defined by
  \begin{align}
    p_0&=s^{(2\sigma+1)/4}\Gamma(\sigma+1/2)^{-1/2}\;,\label{p 0}\\
    p_k&=
      \begin{cases}
        (2ks)^{-1/2}(2sxp_{k-1}-T_\sigma p_{k-1}) & \text{if $k$ is even}\\
        (2(k+2\sigma)s)^{-1/2}(2sxp_{k-1}-T_\sigma p_{k-1}) & \text{if $k$ is odd}\;,
      \end{cases}\label{p k}
  \end{align}
for $k\ge1$. Each $p_k$ is of precise degree $k$, even/odd if $k$ is even/odd, and with positive leading coefficient, denoted by $\gamma_k$. Up to normalization, $p_k$ is the sequence of generalized Hermite polynomials  \cite[p.~380, Problem~25]{Szego1975} and generalized Hermite functions (they are orthogonormal with respect to the measure $|x|^{2\sigma}e^{-sx^2}\,dx$). By~\eqref{p k},
  \begin{equation}\label{gamma k}
    \gamma_k=
      \begin{cases}
        k^{-1/2}(2s)^{1/2}\gamma_{k-1} & \text{if $k$ is even}\\
        (k+2\sigma)^{-1/2}(2s)^{1/2}\gamma_{k-1} & \text{if $k$ is odd}
      \end{cases}
  \end{equation}
for $k\ge1$. We also have $T_\sigma p_0=0$ and
  \begin{equation}
    T_\sigma p_k=
          \begin{cases}
            (2ks)^{1/2}p_{k-1} & \text{if $k$ is even}\\
            (2(k+2\sigma)s)^{1/2}p_{k-1} & \text{if $k$ is odd}\;.
          \end{cases}
        \label{D sigma p k}
  \end{equation}
The following recursion formula follows directly from~\eqref{p k}
and~\eqref{D sigma p k}:
   \begin{equation}\label{recurrence}
     p_k=
       \begin{cases}
         k^{-1/2}\big((2s)^{1/2}xp_{k-1}-(k-1+2\sigma)^{1/2}p_{k-2}\big) & \text{if $k$ is even}\\
         (k+2\sigma)^{-1/2}\big((2s)^{1/2}xp_{k-1}-(k-1)^{1/2}p_{k-2}\big) & \text{if $k$ is odd}\;.
       \end{cases}
   \end{equation}

We have $p_k(0)=0$ if and only if $k$ is odd, and $p_k'(0)=0$ if and
only if $k$ is even. By~\eqref{recurrence} and induction on $k$,
  \begin{equation}\label{p k(0)}
    p_k(0)=(-1)^{k/2}\sqrt{\frac{(k-1+2\sigma)(k-3+2\sigma)\cdots(1+2\sigma)}{k(k-2)\cdots2}}\,p_0
  \end{equation}
if $k$ is even\footnote{As a convention, the product of an empty set
of factors is $1$. In this sense,~\eqref{p k(0)} and~\eqref{p k'(0)} also make sense for $k=0$ and $k=1$, respectively.}. When $k$ is odd, by~\eqref{D sigma p k} and~\eqref{p k(0)},
  \[
    (T_\sigma p_k)(0)=(-1)^{(k-1)/2}\sqrt{\frac{(k+2\sigma)(k-2+2\sigma)\cdots(1+2\sigma)2s}{(k-1)(k-3)\cdots2}}\,p_0\;,
  \]
obtaining by~\eqref{(T sigma m phi)(0)} that
  \begin{equation}\label{p k'(0)}
    p_k'(0)=\frac{(-1)^{(k-1)/2}}{1+2\sigma}\sqrt{\frac{(k+2\sigma)(k-2+2\sigma)\cdots(1+2\sigma)2s}{(k-1)(k-3)\cdots2}}\,p_0\;.
  \end{equation}

The following assertions come from the general theory of orthogonal
polynomials \cite[Chapter~III]{Szego1975}. All zeros of each
polynomial $p_k$ are real and of multiplicity one. Each open
interval between consecutive zeros of $p_k$ contains exactly one
zero of $p_{k+1}$, and at least one zero of every $p_\ell$ with
$\ell>k$. Moreover $p_k$ has exactly $\lfloor k/2\rfloor$ positive
zeros and $\lfloor k/2\rfloor$ negative zeros. The zeros of each
$p_k$ will be denoted $x_{k,1}>x_{k,2}>\dots>x_{k,k}$. On each
interval $(x_{k,i+1},x_{k,i})$, the function $p_{k+1}/p_k$ is
strictly increasing, and satisfies $\lim_{x\to x_{k,i}^\pm}\frac{p_{k+1}(x)}{p_k(x)}=\mp\infty$. For every polynomial $p$ of degree $\le k-1$, we have
  \begin{equation}\label{p^2(x)}
    p^2(x)\le\int_{-\infty}^\infty p^2(t)\,|t|^{2\sigma}e^{-st^2}\,dt\cdot\sum_{\ell=0}^kp_\ell^2(x)
  \end{equation}
for all $x\in\R$. The Gauss-Jacobi formula states
that there are $\lambda_{k,1},\lambda_{k,2},\dots,\lambda_{k,k}\in\R$ such that,
for any polynomial $p$ of degree $\le2k-1$,
  \begin{equation}\label{Gauss-Jacobi}
    \int_{-\infty}^\infty p(x)\,|x|^{2\sigma}e^{-sx^2}\,dx=\sum_{i=1}^kp(x_{k,i})\lambda_{k,i}\;.
  \end{equation}

\begin{lem}\label{l:p k' 2(x k,i) lambda k,i}
  We have
    \[
      {p_k'}^2(x_{k,i})\lambda_{k,i}=
        \begin{cases}
          2s & \text{if $k$ is even}\\
          2s/(1+2\sigma) & \text{if $k$ is odd}\;.
        \end{cases}
    \]
\end{lem}

\begin{proof}
  This is a direct adaptation of the proof of \cite[Corollary~3]{BonanClark1990}. With $p=\frac{p_kp_{k-1}}{x-x_{k,i}}$, the formula~\eqref{Gauss-Jacobi} becomes $\frac{\gamma_k}{\gamma_{k-1}}=p_k'(x_{k,i})p_{k-1}(x_{k,i})\lambda_{k,i}$, and the result follows from~\eqref{gamma k}--\eqref{D sigma p k}.
\end{proof}

\subsection{Relation with the Laguerre polynomials}\label{ss:Laguerre}

For $\alpha>-1$, the sequence of Laguerre polynomials, $L_n^\alpha$, can be determined by requiring the following \cite[Chapter~V, Section~5.1]{Szego1975}:
  \begin{equation}\label{Laguerre}
    \int_0^\infty L_m^\alpha(y)L_n^\alpha(y)y^\alpha e^{-y}\,dy=\frac{\Gamma(n+\alpha+1)}{n!}\,\delta_{mn}\;,
  \end{equation}
each $L_n^\alpha$ has precise degree $n$, and the sign of its leading coefficient is $(-1)^n$. Thus the corresponding Laguerre functions
  \[
    \LL_n^\alpha(y)=\left(\frac{n!}{\Gamma(n+\alpha+1)}\right)^{1/2}L_n^\alpha(y)y^{\alpha/2}e^{-y/2}
  \]
form a complete orthonormal system of $L^2(\R_+,dy)$. With the change of variable $y=sx^2$ in~\eqref{Laguerre}, we get
  \[
    2s^{\alpha+1}\int_0^\infty L_m^\alpha(sx^2)L_n^\alpha(sx^2)x^{2\alpha+1}e^{-sx^2}\,dx=\frac{\Gamma(n+\alpha+1)}{n!}\,\delta_{mn}\;.
  \]
By taking $\sigma=\alpha\pm1/2$, it follows that  (see e.g.\ \cite[p.~525]{Rosler1998} or \cite[p.~23]{Rosler2003})
  \begin{align}
    p_{2n}(x)&=(-1)^n\left(\frac{n!}{\Gamma(n+\sigma+1/2)}\right)^{1/2}s^{\sigma/2+1/4}L_n^{\sigma-1/2}(sx^2)\;,\notag\\
    p_{2n+1}(x)&=(-1)^n\left(\frac{n!}{\Gamma(n+\sigma+3/2)}\right)^{1/2}s^{\sigma/2+3/4}xL_n^{\sigma+1/2}(sx^2)\;,\notag\\
    \xi_{2n}(x)&=(-1)^ns^{\sigma/2+1/4}x^{1/2}\LL_n^{\sigma-1/2}(sx^2)\;,\label{xi_2n}\\
    \xi_{2n+1}(x)&=(-1)^ns^{\sigma/2+3/4}x^{1/2}\LL_n^{\sigma+1/2}(sx^2)\;,\label{xi_2n+1}
  \end{align}
for all $x>0$. On the other hand, the following asymptotic estimates of Laguerre functions were shown in \cite{Erdelyi1960} for $\alpha\ge0$ (see also \cite{AskeyWainger1965,Muckenhoupt1970a-II}), and extended to the case $-1<\alpha<0$ in \cite{Muckenhoupt1970b}: there are some $C,\gamma>0$, depending only on $\alpha$, such that, for all $n\in\N$ and $x>0$,
  \begin{equation}\label{asymptotic estimates of Laguerre}
    |\LL_n^\alpha(x)|\le
      \begin{cases}
        Cx^{\alpha/2}\nu^{\alpha/2} & \text{if $0<x\le1/\nu$}\\
        Cx^{-1/4}\nu^{-1/4} & \text{if $1/\nu<x\le\nu/2$}\\
        Cx^{-1/4}(\nu^{1/3}+|x-\nu|)^{-1/4} & \text{if $\nu/2<x\le3\nu/2$}\\
        Ce^{-\gamma x} & \text{if $3\nu/2<x$\;,}
      \end{cases}
  \end{equation}
where $\nu=4n+2\alpha+2$, with the proviso that we must take $\nu=2$ if $n=0$ and $\alpha<0$. By~\eqref{xi_2n},~\eqref{xi_2n+1} and~\eqref{asymptotic estimates of Laguerre},
  \begin{align}
    |\xi_{2n}(x)|&\le
      \begin{cases}
        Cs^\sigma x^\sigma\nu^{\sigma/2-1/4} & \text{if $0<x\le\sqrt{\frac{1}{s\nu}}$}\\
        Cs^{\sigma/2}\nu^{-1/4} & \text{if $\sqrt{\frac{1}{s\nu}}<x\le\sqrt{\frac{\nu}{2s}}$}\\
        Cs^{\sigma/2}(\nu^{1/3}+|sx^2-\nu|)^{-1/2} & \text{if $\sqrt{\frac{\nu}{2s}}<x\le\sqrt{\frac{3\nu}{2s}}$}\\
        Cs^{\sigma/2+1/4}x^{1/2}e^{-\gamma sx^2} & \text{if $\sqrt{\frac{3\nu}{2s}}<x$\;,}
      \end{cases}\label{|xi_2n|}\\
    \intertext{where $\nu=4n+2\sigma+1$, with the proviso that we must take $\nu=2$ if $n=0$ and $\sigma<1/2$, and}
    |\xi_{2n+1}(x)|&\le
      \begin{cases}
        Cs^{\sigma+1}x^{\sigma+1}\nu^{\sigma/2+1/4} & \text{if $0<x\le\sqrt{\frac{1}{s\nu}}$}\\
        Cs^{(\sigma+1)/2}\nu^{-1/4} & \text{if $\sqrt{\frac{1}{s\nu}}<x\le\sqrt{\frac{\nu}{2s}}$}\\
        Cs^{(\sigma+1)/2}(\nu^{1/3}+|sx^2-\nu|)^{-1/2} & \text{if $\sqrt{\frac{\nu}{2s}}<x\le\sqrt{\frac{3\nu}{2s}}$}\\
        Cs^{\sigma/2+3/4}x^{1/2}e^{-\gamma sx^2} & \text{if $\sqrt{\frac{3\nu}{2s}}<x$\;,}
      \end{cases}\label{|xi_2n+1|}
  \end{align}
where $\nu=4n+2\sigma+3$. Theorems~\ref{t:upper estimates of xi k}-(ii),(iii) and~\ref{t:upper estimates of phi k for k even and sigma<0} easily follow from~\eqref{|xi_2n|} and~\eqref{|xi_2n+1|}.

\section{Estimates using the method of Bonan-Clark}\label{s:estimates}

\subsection{Second perturbation of $H$}\label{ss:K}

Consider the perturbed derivative,
  \[
    E_\sigma=|x|^\sigma T_\sigma|x|^{-\sigma}=
      \begin{pmatrix}
        0 & \frac{d}{dx}+\sigma x^{-1} \\
        \frac{d}{dx}-\sigma x^{-1} & 0
      \end{pmatrix}\;,
  \]
and the perturbed harmonic oscillator,
  \[
    K=|x|^\sigma L|x|^{-\sigma}=-E_\sigma^2+s^2x^2=
      \begin{pmatrix}
       H+\sigma(\sigma-1)x^{-2} & 0 \\
        0 & H+\sigma(\sigma+1)x^{-2}
      \end{pmatrix}\;,
  \]
defined on $|x|^\sigma\,\SS=|x|^\sigma\,\SS_{\text{\rm ev}}\oplus|x|^\sigma\,\SS_{\text{\rm odd}}$. According to Sections~\ref{ss:L} and~\ref{ss:p k}, and since
$|x|^\sigma:L^2(\R,|x|^{2\sigma}\,dx)\to L^2(\R,dx)$ is a unitary
isomorphism, $K$ is essentially self-adjoint in $L^2(\R,dx)$, and
the spectrum of its self-adjoint extension, denoted by $\KK$, or $\KK_\sigma$, consists of the eigenvalues $(2k+1+2\sigma)s$ ($k\in\N$) of multiplicity one, and the corresponding normalized
eigenfunctions are $\xi_k=|x|^\sigma\phi_k=|x|^\sigma p_ke^{-sx^2/2}$.

Each $\xi_k$ is $\Cinf$ on $\R\setminus\{0\}$, and it is
$\Cinf$ on $\R$ if and only if $\sigma\in\N$. If
$\sigma>0$ or $k$ is odd, then $\xi_k$ is defined and continuous on
$\R$, and $\xi_k(0)=0$. If $\sigma<0$ and $k$ is even, then $\xi_k$
is only defined on $\R\setminus\{0\}$; in fact, by~\eqref{p k(0)}, $\lim_{x\to0}\xi_k(x)=(-1)^{k/2}\infty$.

By~\eqref{D sigma p k} and~\eqref{recurrence},
  \begin{align}
    \xi_k'&=\left(p_k'+\left(\frac{\sigma}{x}-sx\right)p_k\right)|x|^\sigma e^{-sx^2/2}\label{xi k'}\\
    &=
      \begin{cases}
        (\sqrt{2ks}\,p_{k-1}+(\frac{\sigma}{x}-sx)p_k)\,|x|^\sigma e^{-sx^2/2} & \text{if $k$ is even}\\
        (\sqrt{2(k+2\sigma)s}\,p_{k-1}-(\frac{\sigma}{x}+sx)p_k)\,|x|^\sigma e^{-sx^2/2} & \text{if $k$ is odd}
      \end{cases}\notag\\
    &=
      \begin{cases}
        ((sx+\frac{\sigma}{x})p_k-\sqrt{2(k+1+2\sigma)s}\,p_{k+1})\,|x|^\sigma e^{-sx^2/2} & \text{if $k$ is even}\\
        ((sx-\frac{\sigma}{x})p_k-\sqrt{2(k+1)s}\,p_{k+1})\,|x|^\sigma e^{-sx^2/2} & \text{if $k$ is
        odd}\;.
      \end{cases}\label{xi k', cases}
  \end{align}
By~\eqref{xi k'},~\eqref{p k(0)} and~\eqref{p k'(0)},
  \[
    \lim_{x\to0^\pm}\xi_k'(x)=
      \begin{cases}
        0 & \text{if $\sigma>1$ or $\sigma=0$}\\
        \pm p_k(0) &\text{if $\sigma=1$}\\
        \pm(-1)^{k/2}\infty &\text{if $0<\sigma<1$}\\
        \mp(-1)^{k/2}\infty &\text{if $-1/2<\sigma<0$}
      \end{cases}
  \]
if $k$ is even,
  \[
    \lim_{x\to0}\xi_k'(x)=
      \begin{cases}
        0 & \text{if $\sigma>0$}\\
        p_k'(0) &\text{if $\sigma=0$}\\
        (-1)^{(k-1)/2}\infty &\text{if $-1/2<\sigma<0$}
      \end{cases}
  \]
if $k$ is odd, and
  \begin{equation}\label{lim x to0 pm(xi k xi k')(x)}
    \lim_{x\to0^\pm}(\xi_k\xi_k')(x)=
      \begin{cases}
        0 & \text{if $k$ is odd or $\sigma=0$ or $\sigma>1/2$}\\
        \pm p_k^2(0)/2 & \text{if $k$ is even and $\sigma=1/2$}\\
        \pm\infty & \text{if $k$ is even and $0<\sigma<1/2$}\\
        \mp\infty & \text{if $k$ is even and $-1/2<\sigma<0$}\;.
      \end{cases}
  \end{equation}
By~\eqref{xi k', cases},
  \begin{equation}\label{frac xi k' xi k}
    \frac{\xi_k'}{\xi_k}=
      \begin{cases}
        sx+\frac{\sigma}{x}-\sqrt{2(k+1+2\sigma)s}\,\frac{p_{k+1}}{p_k} & \text{if $k$ is even}\\
        sx-\frac{\sigma}{x}-\sqrt{2(k+1)s}\,\frac{p_{k+1}}{p_k} & \text{if $k$ is odd}\;,
      \end{cases}
  \end{equation}
which generalizes a formula of \cite{Hille1926} for the Hermite functions.

For the sake of simplicity, let $\bar\sigma_k=\sigma(\sigma-(-1)^k)$. Each $\xi_k$ satisfies
  \begin{equation}\label{xi k''+q k xi k=0}
    \xi_k''+q_k\xi_k=0\;,
  \end{equation}
where $q_k=(2k+1+2\sigma)s-s^2x^2-\bar\sigma_kx^{-2}$.

\subsection{Description of $q_k$}\label{ss:q k}

The following elementary analysis of the functions $q_k$ will be used in Sections~\ref{ss:zeros} and~\ref{ss:estimates of xi k}. If $k$ is even, then $\bar\sigma_k=0$ if $\sigma\in\{0,1\}$, $\bar\sigma_k<0$ if $0<\sigma<1$, and $\bar\sigma_k>0$ otherwise. When $k$ is odd, we have $\bar\sigma_k=0$ if $\sigma=0$, and $\sigma\bar\sigma_k>0$ if $\sigma\ne0$. Each $q_k$ is defined and smooth on $\R$ just when $\bar\sigma_k=0$, otherwise it is defined and smooth only on $\R\setminus\{0\}$. Moreover $q_k$ is even and $q_k'=-2s^2x+2\bar\sigma_kx^{-3}$. Observe that
  \begin{alignat*}{2}
    \lim_{x\to\pm\infty}q_k(x)&=-\infty\;,&\quad
    \lim_{x\to0}q_k(x)&=
      \begin{cases}
        -\infty & \text{if $\bar\sigma_k>0$}\\
        \infty & \text{if $\bar\sigma_k<0$}\;,
      \end{cases}\\
    \lim_{x\to\pm\infty}q_k'(x)&=\mp\infty\;,&\quad
    \lim_{x\to0^{\pm}}q_k'(x)&=
      \begin{cases}
        \pm\infty & \text{if $\bar\sigma_k>0$}\\
        \mp\infty & \text{if $\bar\sigma_k<0$}\;.
      \end{cases}
  \end{alignat*}
We have the following cases for the zeros of $q_k'$:
  \begin{itemize}

    \item If $\bar\sigma_k>0$, then $q_k'$ has two zeros, $\pm x_{\text{max}}=\pm\sqrt{\sqrt{\bar\sigma_k}/s}$. At these points, $q_k$ reaches its maximum, which equals $c_{\text{max}}s$ for $c_{\text{max}}=2k+1+2\sigma-2\sqrt{\bar\sigma_k}$. Notice that, in this case, $c_{\text{max}}=0$ if $k=0$ and $\sigma=-1/8$, $c_{\text{max}}<0$ if $k=0$ and $-1/2<\sigma<-1/8$, and $c_{\text{max}}>0$ otherwise.

    \item If $\bar\sigma_k=0$, then $q_k'$ has one zero, which is $0$, where $q_k$ reaches its maximum $c_{\text{max}}s$ as above with $c_{\text{max}}=2k+1+2\sigma>0$.

    \item If $\bar\sigma_k<0$, then $q_k'>0$ on $\R_-$ and $q_k'<0$ on $\R_+$.

\end{itemize}
We have the following possibilities for the zeros of $q_k$:
  \begin{itemize}

    \item If $\bar\sigma_k>0$ and $c_{\text{max}}>0$, then $q_k$ has four zeros, which are
      \begin{align}
        \pm a_k&=\pm\sqrt{\frac{2k+1+2\sigma-\sqrt{(2k+1+2\sigma)^2-4\bar\sigma_k}}{2s}}\;,
        \notag\\
        \pm b_k&=\pm\sqrt{\frac{2k+1+2\sigma+\sqrt{(2k+1+2\sigma)^2-4\bar\sigma_k}}{2s}}
        \label{b_k}\;.
       \end{align}

    \item If $\bar\sigma_k>0$ and $c_{\text{max}}=0$, or $\bar\sigma_k\le0$, then $q_k$ has two zeros, $\pm b_k$,  defined by~\eqref{b_k}.

    \item If $\bar\sigma_k>0$ and $c_{\text{max}}<0$, then $q_k<0$.

  \end{itemize}

If $q_k$ has four zeros, $\pm a_k$ and $\pm b_k$, then
  \begin{equation}\label{s(b k-a k) 2}
    s(b_k-a_k)^2=c_{\text{max}}\;,
  \end{equation}
and
  \[
    2sa_k^2=\frac{4\bar\sigma_k}{2k+1+2\sigma+\sqrt{(2k+1+2\sigma)^2-4\bar\sigma_k}}\;,
  \]
obtaining
      \begin{equation}\label{a_k in O(k^-1/2)}
        a_k\in O(k^{-1/2})
      \end{equation}
as $k\to\infty$. 

If $q_k$ has at least two zeros, $\pm b_k$, then  
  \[
    2s(b_k^2-b_\ell^2)=2+4\,\frac{k^2-\ell^2+(1+2\sigma)(k-\ell)+\bar\sigma_\ell-\bar\sigma_k}{\sqrt{(2k+1+2\sigma)^2-4\bar\sigma_k}+\sqrt{(2\ell+1+2\sigma)^2-4\bar\sigma_\ell}}
  \]
for $\ell\le k$, obtaining
  \begin{equation}\label{b_k+1 -b_k in O(k^-1/2)}
    b_{k+1}-b_k\in O(k^{-1/2})
  \end{equation}
as $k\to\infty$, and
  \begin{equation}\label{b_k-b_ell ge C(k-ell)k^-1/2}
    b_k-b_\ell\ge C(k-\ell)k^{-1/2}
  \end{equation}
for some $C>0$ if $k$ and $\ell$ are large enough. If $\bar\sigma_k=0$, then $sb_k^2=c_{\text{max}}$.

Like in~\cite{BonanClark1990}, the maximal open intervals where $q_k$ is defined and $>0$
(respectively, $<0$) will be called {\em oscillation\/}
(respectively, {\em non-oscillation\/}) intervals of $\xi_k$; this
terminology is justified by Lemma~\ref{l:zeros in oscillation
intervals} below. We have the following possibilities for the oscillation intervals:
  \begin{itemize}

    \item If $\bar\sigma_k>0$ and $c_{\text{max}}>0$, then $\xi_k$ has two oscillation intervals, $(a_k,b_k)$ and $(-b_k,-a_k)$, containing $x_{\text{max}}$ and $-x_{\text{max}}$, respectively.

    \item If $\bar\sigma_k>0$ and $c_{\text{max}}\le0$, then $\xi_k$ has no oscillation intervals.

    \item If $\bar\sigma_k<0$, then $\xi_k$ has two oscillation
intervals, $(-b_k,0)$ and $(0,b_k)$.

    \item If $\bar\sigma_k=0$, then $\xi_k$ has one oscillation interval, $(-b_k,b_k)$.

  \end{itemize}

\subsection{Location of the zeros of $\xi_k$ and $\xi_k'$}\label{ss:zeros}

In $\R\setminus\{0\}$, the functions $\xi_k$ and $p_k$ have the same
zeros. Then $\xi_k$ and $\xi_k'$ have no common zeros by~\eqref{xi
k'}. The functions $\xi_0$ and $\xi_1$ have no zeros in
$\R\setminus\{0\}$, and the two zeros $\pm x_{2,1}$ of $\xi_2$ are
in $\R\setminus\{0\}$.

\begin{lem}\label{l:zeros in oscillation intervals}
  On $\R\setminus\{0\}$:
    \begin{itemize}

      \item[(i)] the zeros of $\xi_k'$ belong to the oscillation intervals of $\xi_k$;

      \item[(ii)] if $k$ is odd or $\sigma\ge0$, the zeros of $\xi_k$ belong to the oscillation intervals of $\xi_k$;

      \item[(iii)] if $k$ is even and $\sigma<0$, the zeros of $\xi_k$, possibly except $\pm x_{k,k/2}$, belong to the oscillation
      intervals of $\xi_k$.

    \end{itemize}
\end{lem}

\begin{proof}
  It is enough to consider the zeros in $\R_+$ because $\xi_k$ is either even or odd. We can also assume that $\xi_k\xi_k'$ has zeros on $\R_+$, otherwise there is nothing to prove.

  Let $x_*$ and $x^*$ denote the minimum and maximum of the zeros of $\xi_k\xi_k'$ in $\R_+$. By \eqref{xi k''+q k xi k=0}, we have $(\xi_k\xi_k')'={\xi_k'}^2-q_k\xi_k^2>0$ on the non-oscillation intervals, and therefore $\xi_k\xi_k'$ is strictly increasing on those
  intervals. In particular, since $\xi_k\xi_k'$ is strictly increasing on
  $(b_k,\infty)$ and $(\xi_k\xi_k')(x)\to0$ as $x\to\infty$, it follows that $x^*<b_k$. This shows the statement when there is
  one oscillation interval of the form $(-b_k,b_k)$. So it remains to consider
  the case where there is an oscillation interval of $\xi_k$ in $\R_+$ of the form $(a_k,b_k)$. This holds when $k$ is odd and $\sigma>0$, $k=0$ and $\sigma\in(-1/8,0)\cup(1,\infty)$, or $k\in2\Z_+$
  and $\sigma\in(-1/2,0)\cup(1,\infty)$.

  If $k$ is odd and $\sigma>0$, or $k$ is even and $\sigma\in(1,\infty)$, then $x_*>
  a_k$ because $\xi_k\xi_k'$ is strictly increasing on
  $(0,a_k)$ and $(\xi_k\xi_k')(x)\to0$ as $x\to0^+$ by~\eqref{lim x to0 pm(xi k xi k')(x)}.

  Finally, assume that $k\in2\Z_+$ and $\sigma\in(-1/2,0)$. Then the above arguments do not work because $(\xi_k\xi_k')(x)\to-\infty$ as
  $x\to0^+$ by~\eqref{lim x to0 pm(xi k xi k')(x)}. Let $f$ be the function on $\R_+$ defined by $f(x)=sx+\frac{\sigma}{x}$. We have $f(x)\to-\infty$
  as $x\to0^+$, and $f'=s-\frac{\sigma}{x^2}>0$ on $\R_+$. Moreover $\sqrt{-\sigma/s}$ is the unique zero of $f$ in $\R_+$.

  If $x_*$ is a zero of $\xi_k'$, then $\xi_k$ (and $p_k$ as well) has no zeros in $[-x_*,x_*]$. Therefore $0$ is the unique zero of $p_{k+1}$ in this interval. So $p_{k+1}/p_k>0$ on $(0,x_*]$. Since
    \[
      0=f(x_*)-\sqrt{2(k+1+2\sigma)s}\,\frac{p_{k+1}(x_*)}{p_k(x_*)}
    \]
  by~\eqref{frac xi k' xi k}, it follows that $f(x_*)>0$, obtaining
  $x_*>\sqrt{-\sigma/s}$. But
    \[
      a_k^2=\frac{2k+1+2\sigma-\sqrt{(2k+1)^2+8(k+1)\sigma}}{2s}<-\frac{\sigma}{s}
    \]
  because $k>1$, obtaining $x_*>a_k$.

  If $x_*$ is a zero of $\xi_k$ (i.e., $x_*=x_{k,k/2}$), then the other positive zeros of $\xi_k\xi_k'$ are $>a_k$ because
  this function is strictly increasing on $(0,a_k)$.
\end{proof}

 In the case of Lemma~\ref{l:zeros in oscillation intervals}-(iii), the zero $\pm x_{k,k/2}$ of $\xi_k$ may be in an oscillation interval, in a non-oscillation interval, or in their common boundary point. For instance, for $k=2$,
    \[
      p_2=\left(\sqrt{\frac{2}{1+2\sigma}}sx^2-\sqrt{\frac{1+2\sigma}{2}}\right)p_0
    \]
  by~\eqref{p k}, obtaining $x_{2,1}^2=\frac{1+2\sigma}{2s}$. Moreover
    \[
      a_2^2=\frac{5+2\sigma-\sqrt{25+24\sigma}}{2s}\;.
    \]
  So
    \[
      x_{2,1}-a_2=\frac{-4+\sqrt{25+24\sigma}}{2s}\;,
    \]
  and therefore $\sigma>-3/8$ if and only if $x_{2,1}>a_2$. So $(a_2,b_2)$ contains no zero of $\xi_2$ when $\sigma\in(-1/2,-3/8]$. For $k>2$, every oscillation interval
of $\xi_k$ contains some zero of $\xi_k$ by Lemma~\ref{l:zeros in
oscillation intervals}.

\begin{lem}\label{l:J}
  There are $C_0,C_1,C_2>0$, depending on $\sigma$ and $s$, such that, if $k\ge C_0$ and $I$ is any oscillation interval of $\xi_k$, then there is some subinterval $J\subset I$ so that:
    \begin{itemize}

      \item[(i)] for every $x\in J$, there is some zero $x_{k,i}$ of
      $\xi_k$ in $I$ such that $|x-x_{k,i}|\le\frac{C_1}{\sqrt{q_k(x)}}$;

      \item[(ii)]  each connected component of $I\setminus J$ is of length
      $\le C_2k^{-1/2}$.

    \end{itemize}
\end{lem}

\begin{proof}
  According to Section~\ref{ss:q k}, for any $c>0$ with $cs\in q_k(I)$, the set
  $I_c=I\cap q_k^{-1}([cs,\infty))$ is a subinterval of $I$, whose boundary in $I$ is $I\cap
  q_k^{-1}(cs)$.

    \begin{claim}\label{cl:distribution of the zeros}
      If $\length(I_c)\ge2\pi/\sqrt{cs}$, then each boundary point of $I_c$ in $I$ satisfies the condition of~(i) with $x_{k,i}\in I_c$ and $C_1=2\pi$.
    \end{claim}

  Let $f_c$ be the function on $\R$ defined by $f_c(x)=\sin(\sqrt{cs}\,x)$, whose zeros are $\ell\pi/\sqrt{cs}$ for $\ell\in\Z$. Since $f_c''+csf_c=0$ and $cs\le q_k$ on $I_c$, the zeros of $\xi_k$ in $I_c$ separate the zeros of $f_c$ in $I_c$ by Sturm's comparison theorem. If $\length(I_c)\ge2\pi/\sqrt{cs}$, then each boundary point $x$ of $I_c$ is at a distance $\le2\pi/\sqrt{cs}$ of two consecutive zeros of $f_c$ in $I_c$, and there is some zero of $\xi_k$ between them, which shows Claim~\ref{cl:distribution of the zeros} because $q_k(x)=cs$.

  Now we have to analyze each type of oscillation interval separately, corresponding to the possibilities for $\bar\sigma_k$ and $c_{\text{max}}$. When there are two oscillation intervals of $\xi_k$, it is enough to consider only the oscillation interval contained in $\R_+$ because the function $\xi_k$ is either even or odd.

  The first type of oscillation interval is of the form $I=(a_k,b_k)$, which corresponds to the conditions $\bar\sigma_k>0$ and $c_{\text{max}}>0$. We have $cs\in q_k(I)$ when $0<c\le c_{\text{max}}$. Then $q_k^{-1}(cs)$ consists of the points
    \begin{align}
      \pm a_{k,c}&=\pm\sqrt{\frac{2k+1+2\sigma-c-\sqrt{(2k+1+2\sigma-c)^2-4\bar\sigma_k}}{2s}}\;,
      \notag\\
      \pm b_{k,c}&=\pm\sqrt{\frac{2k+1+2\sigma-c+\sqrt{(2k+1+2\sigma-c)^2-4\bar\sigma_k}}{2s}}\;,
      \label{b k,c}
    \end{align}
  and we get $I_c=[a_{k,c},b_{k,c}]$. Since
    \begin{equation}\label{s(b k,c-a k,c) 2}
      s(b_{k,c}-a_{k,c})^2=c_{\text{max}}-c\;,
    \end{equation}
  we have $\length(I_c)\ge2\pi/\sqrt{cs}$ if and only if $c(c_{\text{max}}-c)\ge4\pi^2$, which means that $c_{\text{max}}\ge4\pi$ and $c_-\le c\le c_+$ for
    \[
      c_\pm=\frac{c_{\text{max}}\pm\sqrt{c_{\text{max}}^2-16\pi^2}}{2}\;.
    \]
  Since $c_{\text{max}}\in O(k)$ as $k\to\infty$, there is some $C_0>0$, depending on $\sigma$ and $s$, such that $c_{\text{max}}\ge4\pi$ for all $k\ge C_0$. Assuming $k\ge C_0$, let $a_{k,\pm}=a_{k,c_\pm}$ and
  $b_{k,\pm}=b_{k,c_\pm}$, which satisfy
    \[
      a_k<a_{k,-}<a_{k,+}<b_{k,+}<b_{k,-}<b_k\;.
    \]

  Fix any $x\in I$ and let $q_k(x)=cs$. First, $x\in[a_{k,-},a_{k,+}]\cup[b_{k,+},b_{k,-}]$ if and only if $\length(I_c)\ge2\pi/\sqrt{cs}$, and in this case $x$ satisfies the condition of~(i) with $x_{k,i}\in I_c$ and $C_1=2\pi$ by Claim~\ref{cl:distribution of the zeros}. Second, if $x\in(a_k,a_{k,-})\cup(b_{k,-},b_k)$, then $\length(I_c)<2\pi/\sqrt{cs}$, $I_c\supset I_{c_-}$, and we already know that $I_{c_-}$ contains some zero of $\xi_k$. Hence $x$ also satisfies the condition of~(i) with $C_1=2\pi$. And third, if $x\in(a_{k,+},b_{k,+})$, then
    \begin{gather*}
      s(b_{k,+}-a_{k,+})^2=c_{\text{max}}-c_+=c_-=\frac{16\pi^2}{c_+}\le\frac{32\pi^2}{c_{\text{max}}}\le\frac{32\pi^2}{c}
    \end{gather*}
   by~\eqref{s(b k,c-a k,c) 2}, obtaining $\length(I_{c_+})\le4\sqrt{2}\pi/\sqrt{cs}$. Since $I_c\subset I_{c_+}$ and it is already proved that $I_{c_+}$ contains some zero of $\xi_k$, it follows that $x$ also satisfies the condition of~(i) with $C_1=4\sqrt{2}\pi$. Summarizing,~(i) holds in this case with $J=I$ and $C_1=4\sqrt{2}\pi$ if $c_{\text{max}}\ge4\pi$. In this case,~(ii) is obvious because $J=I$.

  The second type of oscillation interval is of the form $I=(0,b_k)$, which corresponds to the condition $\bar\sigma_k<0$. Now, $cs\in q_k(I)$ for any $c>0$, the set $q_k^{-1}(cs)$ consists of the points $\pm b_{k,c}$, defined like in~\eqref{b k,c}, and we have $I_c=(0,b_{k,c}]$. The equality $cs=q_k(2\pi/\sqrt{cs})$ holds when
    \begin{equation}\label{(2k+1+sigma) 2-bar sigma k-16 pi 2>0}
      (2k+1+2\sigma)^2-4\bar\sigma_k-16\pi^2>0
    \end{equation}
  and $c$ is
    \[
      c_\pm=2\pi^2\,\frac{2k+1+2\sigma\pm\sqrt{(2k+1+2\sigma)^2-4\bar\sigma_k-16\pi^2}}{\bar\sigma_k-4\pi^2}\;.
    \]
  Assuming~\eqref{(2k+1+sigma) 2-bar sigma k-16 pi 2>0}, we have $\length(I_c)\ge2\pi/\sqrt{cs}$ if and only if $c_-\le c\le c_+$. Let $b_{k,\pm}=b_{k,c_\pm}$, satisfying $0<b_{k,+}<b_{k,-}<b_k$.

  Fix any $x\in I$ and let $q_k(x)=cs$. First, $x\in[b_{k,+},b_{k,-}]$ if and only if $\length(I_c)\ge2\pi/\sqrt{cs}$; in this case, $x$ satisfies the condition of~(i) with $x_{k,i}\in I_c$ and $C_1=2\pi$ by Claim~\ref{cl:distribution of the zeros}. And second, if $x\in(b_{k,-},b_k)$, then $\length(I_c)<2\pi/\sqrt{cs}$, $I_c\supset I_{c_-}$, and we already know that $I_{c_-}$ contains some zero of $\xi_k$. Hence $x$ also satisfies the condition of~(i) with $C_1=2\pi$. So, when~\eqref{(2k+1+sigma) 2-bar sigma k-16 pi 2>0} is true,~(i) holds with $J=[b_{k,+},b_k)$ and $C_1=2\pi$.

  Notice that  $c_+\in O(k)$  as $k\to\infty$. Then there are some $C_0,C_2>0$, depending on $\sigma$ and $s$, such that, if $k\ge C_0$, then~\eqref{(2k+1+sigma) 2-bar sigma k-16 pi 2>0} holds and $sb_{k,+}^2=4\pi^2/c_+\le C_2k^{-1}$, showing~(ii) in this case.

  The third and final type of oscillation interval is $I=(-b_k,b_k)$, which corresponds to the condition $\bar\sigma_k=0$. We have $cs\in q_k(I)$ when $0<c\le c_{\text{max}}$. Then $q_k^{-1}(cs)$ consists of the points $\pm b_{k,c}$, defined like in~\eqref{b k,c}, and  we get $I_c=[-b_{k,c},b_{k,c}]$. Since
    \begin{equation}\label{sb k,c 2}
      sb_{k,c}^2=c_{\text{max}}-c\;,
    \end{equation}
  we have $\length(I_c)\ge2\pi/\sqrt{cs}$ if and only if $c(c_{\text{max}}-c)\ge\pi^2$, which means that $c_{\text{max}}\ge\pi$ and $c_-\le c\le c_+$ for
    \[
      c_\pm=\frac{c_{\text{max}}\pm\sqrt{c_{\text{max}}^2-4\pi^2}}{2}\;.
    \]
  Since $c_{\text{max}}\in O(k)$ as $k\to\infty$, there is some $C_0>0$, depending on $\sigma$ and $s$, such that $c_{\text{max}}\ge4\pi$ for all $k\ge C_0$. Assuming $k\ge C_0$, let $b_{k,\pm}=b_{k,c_\pm}$, which satisfy $0<b_{k,+}<b_{k,-}<b_k$.

  Fix any $x\in I$ and let $q_k(x)=cs$. First, $b_{k,+}\le |x|\le b_{k,-}$ if and only if $\length(I_c)\ge2\pi/\sqrt{cs}$; in this case, $x$ satisfies the condition of~(i) with $x_{k,i}\in I_c$ and $C_1=2\pi$ by Claim~\ref{cl:distribution of the zeros}. Second, if $|x|>b_{k,-}$, then $\length(I_c)<2\pi/\sqrt{cs}$, $I_c\supset I_{c_-}$, and we already know that $I_{c_-}$ contains some zero of $\xi_k$. Hence $x$ also satisfies the condition of~(i) with $C_1=2\pi$. And third, if $|x|<b_{k,+}$, then
    \[
      sb_{k,+}^2=c_{\text{max}}-c_+=c_-=\frac{4\pi^2}{c_+}\le\frac{8\pi^2}{c_{\text{max}}}\le\frac{8\pi^2}{c}
    \]
   by~\eqref{sb k,c 2}, obtaining $\length(I_{c_+})\le\sqrt{2}\,\pi/\sqrt{cs}$. Since $I_c\subset I_{c_+}$ and it is already proved that $I_{c_+}$ contains some zero of $\xi_k$, it follows that $x$ also satisfies the condition of~(i) with $C_1=\sqrt{2}\pi$. Summarizing,~(i) holds in this case with $J=I$ and $C_1=2\pi$. In this case,~(ii) is also obvious because $J=I$.
\end{proof}

\begin{lem}\label{l:J'}
  For some $C'_0,C'_1,C'_2>0$, depending on $\sigma$ and $s$, if $k\ge C'_0$ and $I$ is any oscillation interval of $\xi_k$, then there is some subinterval $J'\subset I$ so that:
    \begin{itemize}

      \item[(i)] $q_k\ge C'_1k^{1/3}$ on $J'$;

      \item[(ii)] each connected component of $I\setminus J'$ is of length $\le C'_2k^{-1/6}$.

    \end{itemize}
\end{lem}

\begin{proof}
   We  use the notation of the proof of Lemma~\ref{l:J}. The same type of argument can be used for all types of oscillation intervals. Thus, e.g., suppose that $I$ is of the type $(0,b_k)$. Since $b_k\in O(k^{1/2})$ as $k\to\infty$, we have $b'_k=b_k-k^{-1/6}\in I$ for $k$ large enough, and
    \[
      q_k(b'_k)=-s^2(k^{-1/3}-2b_kk^{-1/6})-4\bar\sigma_k((b_k-k^{-1/6})^{-2}-b_k^{-2})\in O(k^{1/3})
    \]
  as $k\to\infty$. So there are $C'_0,C'_1>0$, depending on $\sigma$ and $s$, such that $b'_k\in I$ and $c'=q_k(b'_k)\ge C'_1k^{1/3}$ for $k\ge C'_0$. Then~(i) and~(ii) hold with $J'=I_{c'}=(0,b'_k]$.
\end{proof}

\begin{cor}\label{c:distribution of the zeros}
  There exist $C''_0,C''_1>0$, depending on $\sigma$ and $s$, such that, if $k\ge C''_0$ and $I$ is any oscillation interval of $\xi_k$, then, for each $x\in I$, there exists some zero $x_{k,i}$ of $\xi_k$ in $I$ so that $|x-x_{k,i}|\le C''_1k^{-1/6}$.
\end{cor}

\begin{proof}
  With the notation of Lemmas~\ref{l:J} and~\ref{l:J'}, let $C''_0=\max\{C_0,C'_0\}$ and $C''_2=\max\{C_2,C'_2\}$. Assume $k\ge C''_0$ and consider the subinterval $J''=J\cap J'\subset I$. By Lemmas~\ref{l:J}-(ii) and~\ref{l:J'}-(ii), each connected component of $I\setminus J''$ is of length $\le C''_2k^{-1/6}$. Then, for each $x\in I$, there is some $x''\in J''$ such that $|x-x''|\le C''_2k^{-1/6}$. By Lemmas~\ref{l:J}-(i) and~\ref{l:J'}-(i), there is some zero $x_{k,i}$ of $\xi_k$ in $I$ such that $|x''-x_{k,i}|=\frac{C_1}{\sqrt{q_k(x'')}}\le\frac{C_1}{\sqrt{C'_1}}\,k^{-1/6}$. Hence $|x-x_{k,i}|\le\bigl(C''_2+C_1/\sqrt{C'_1}\bigr)k^{-1/6}$.
\end{proof}

\subsection{Estimates of $\xi_k$}\label{ss:estimates of xi k}

\begin{lem}\label{l:estimates of xi k}
  Let $I$ be an oscillation interval of $\xi_k$, let  $x\in I$ and let $x_{k,i}$ be a zero of $\xi_k$ in $I$. Then
    \[
      \xi_k^2(x)\le
        \begin{cases}
          \frac{8s}{3}\,|x-x_{k,i}| & \text{if $k$ is even}\\
          \frac{8s}{3(1+2\sigma)}\,|x-x_{k,i}| & \text{if $k$ is odd}\;.
        \end{cases}
    \]
\end{lem}

\begin{proof}
   We can assume that there are no zeros of $\xi_k$ between $x$ and $x_{k,i}$. For the sake of simplicity, suppose also that $x_{k,i}<x$ and $\xi_k>0$ on $(x_{k,i},x)$; the other cases are analogous. The key observation of~\cite{BonanClark1990} is that then the graph of $\xi_k$ on $[x_{k,i},x]$ is concave down, and therefore $\frac{1}{2}\xi_k(x)(x-x_{k,i})\le\int_{x_{k,i}}^x\xi_k(t)\,dt$. By Schwartz's inequality and~\eqref{Gauss-Jacobi}, it follows that
    \begin{align*}
      \left(\frac{1}{2}\xi_k(x)(x-x_{k,i})\right)^2&\le\left(\int_{-\infty}^\infty\frac{p_k^2(t)\,|t|^{2\sigma}e^{-st^2}}{(t-x_{k,i})^2}\,dt\right)\left(\int_{x_{k,i}}^x(t-x_{k,i})^2\,dt\right)\\
      &={p_k'}^2(x_{k,i})\,\lambda_{k,i}\,\frac{(x-x_{k,i})^3}{3}\;,
    \end{align*}
  and the result follows by Lemma~\ref{l:p k' 2(x k,i) lambda k,i}.
\end{proof}

With the notation of Lemma~\ref{l:J}, for each $k\ge C_0$, let $\widehat{I}_k$ denote the union of the oscillation intervals of $\xi_k$, and let $\widehat{J}_k\subset\widehat{I}_k$ denote the union of the corresponding subintervals $J$ defined in the proof of Lemma~\ref{l:J}. More precisely:
  \begin{itemize}

    \item if $\bar\sigma_k>0$ and $c_{\text{max}}>0$, then $\widehat{J}_k=\widehat{I}_k=(-a_k,-b_k)\cup(a_k,b_k)$;

    \item if $\bar\sigma_k<0$, then $\widehat{I}_k=(-b_k,0)\cup(0,b_k)$ and $\widehat{J}_k=(-b_k,-b_{k,+}]\cup[b_{k,+},b_k)$;

    \item if $\bar\sigma_k=0$, then $\widehat{J}_k=\widehat{I}_k=(-b_k,b_k)$.

  \end{itemize}
If $k<C_0$, we also use the notation $\widehat{J}_k=\widehat{I}_k$ for the union of the oscillation intervals, which may be empty if there are no oscillation intervals.

\begin{proof}[Proof of Theorem~\ref{t:upper estimates of xi k}]
  Part~(i) follows from Lemmas~\ref{l:J} and~\ref{l:estimates of xi k}.

  In any case, $\xi_k(x)\to0$ as $x\to\infty$. If moreover $k$ is odd or $\sigma\ge0$, then $\xi_k$ is continuous on $\R$. Thus $\xi_k^2$ is bounded and reaches its maximum at some point $\bar x\in\R$. Since $\xi_k(0)=0$ (if $\bar\sigma_k\neq0$) or $0\in\widehat{I}_k$ (if $\bar\sigma_k=0$), it follows from Lemma~\ref{l:zeros in oscillation intervals} that $\bar x\in\widehat{I}_k$. Then~(ii) follows by Corollary~\ref{c:distribution of the zeros} and Lemma~\ref{l:estimates of xi k}.

  If $k$ is even and $\sigma<0$, then $\xi_k$ is not defined at $0$ and $\xi_k^2(x)\to\infty$ as $x\to0$. So we can only conclude as above that the restriction of $\xi_k^2$ to the set defined by $|x|\ge x_{k,k/2}$ is bounded, and reaches its maximum at some point $\bar x$ of this set. Then $\bar x\in\widehat{I}_k$ by Lemma~\ref{l:zeros in oscillation intervals}, and therefore~(iii) holds by Corollary~\ref{c:distribution of the zeros} and Lemma~\ref{l:estimates of xi k}.
\end{proof}

Consider the case $\sigma<0$ and $k$ even, when
Theorem~\ref{t:upper estimates of xi k} does not provide any estimate of
$\xi_k^2$ around zero. According to Section~\ref{ss:p k}, the
function $p_k^2(x)$ on the region $|x|\le x_{k,k/2}$ reaches its
maximum at $x=0$, and moreover $p_k^2(0)<p_0^2$ by~\eqref{p k(0)}.
Hence $\phi_k^2(x)<p_0^2$ for $|x|\le x_{k,k/2}$, which complements
Theorem~\ref{t:upper estimates of xi k}-(iii). On the other hand,
$\phi_k^2(x)\le\xi_k^2(x)$ for $|x|\le1$. Moreover $x_{k,k/2}\le1$
for $k$ large enough by Corollary~\ref{c:distribution of the zeros}
since $a_k\to0$ as $k\to\infty$. So Theorem~\ref{t:upper estimates of xi k}-(iii) implies Theorem~\ref{t:upper estimates of phi k for k even and sigma<0}.

The following lemmas will be used in the proof of Theorem~\ref{t:lower estimates of max xi k^2}.
  
\begin{lem}\label{l:F}
  There exists $F>0$ so that $\xi_k(x)\le\frac{Fk^{-5/12}}{(x-b_k)^2}$ for $k\ge1$ and $x\ge b_{k+1}$.
\end{lem}

\begin{proof}
  Let $x_0\in(x_{k,1},b_k)$ such that $\xi_k'(x_0)=0$. Since $\xi_k'(x)=\int_{x_0}^x\xi_k''(t)\,dt$ and $\xi_k'(x)<0$ for $x>b_k$, we get $\int_{x_0}^xq_k(t)\xi_k(t)\,dt>0$ for $x>b_k$. Because $\xi_k(x)>0$ for $x>x_0$, $q_k(x)>0$ for $x_0<x<b_k$ and $q_k(x)<0$ for $x>b_k$, it follows that
    \begin{equation}\label{int_x_0^b_k q_k(t) xi_k(t) dt>-int_b_k^x q_k(t) xi_k(t) dt}
      \int_{x_0}^{b_k}q_k(t)\xi_k(t)\,dt>-\int_{b_k}^xq_k(t)\xi_k(t)\,dt\;.
    \end{equation}
  
  According to Corollary~\ref{c:distribution of the zeros} and Theorem~\ref{t:upper estimates of xi k}-(ii),(iii), for $k\ge C_0''$ and with $\bar C=\max\{C',C''\}$, we get
    \begin{align*}
      \int_{x_0}^{b_k}q_k(t)\xi_k(t)\,dt&\le{\bar C}^{1/2}k^{-1/12}\int_{x_0}^{b_k}q_k(t)\,dt\\
      &={\bar C}^{1/2}k^{-1/12}\biggl((2k+1+2\sigma)s(b_k-x_0)\\
      &\phantom{=\text{}}\text{}-\frac{s^2}{3}(b_k^3-x_0^3)+\bar\sigma_k(b_k^{-1}-x_0^{-1})\biggr)\\
      &\le{\bar C}^{1/2}k^{-1/12}\biggl((2k+1+2\sigma)sC_1''k^{-1/6}\\
      &\phantom{=\text{}}\text{}-\frac{s^2}{3}(b_k^3-(b_k-C_1''k^{-1/6})^3)+\frac{|\bar\sigma_k|\,C_1''k^{-1/6}}{b_k(b_k-C_1''k^{-1/6})}\biggr)\\
      &\le{\bar C}^{1/2}k^{-1/12}\biggl((2k+1+2\sigma)sC_1''k^{-1/6}\\
      &\phantom{=\text{}}\text{}-s^2\biggl(C_1''b_k^2k^{-1/6}-{C_1''}^2b_kk^{-1/3}-\frac{{C_1''}^3k^{-1/2}}{3}\biggr)\\
      &\phantom{=\text{}}\text{}+\frac{|\bar\sigma_k|\,C_1''k^{-1/6}}{b_k(b_k-C_1''k^{-1/6})}\biggr)\;.
    \end{align*}
  Since $2k+1+2\sigma-sb_k^2=\frac{\bar\sigma_k}{sb_k^2}$, there is some $F_0>0$ such that, for all $k\in\N$,
    \begin{equation}\label{F_0}
      \int_{x_0}^{b_k}q_k(t)\xi_k(t)\,dt\le F_0k^{1/12}\;.
    \end{equation}
  
  On the other hand, 
    \[
      -\int_{b_k}^xq_k(t)\xi_k(t)\,dt\ge-\xi_k(x)\int_{b_k}^xq_k(t)\,dt\;.
    \]
  With the substitution $u=t-b_k$, we get
    \[
      q_k(t)=-s^2u(u+2b_k)+\frac{\bar\sigma_k}{b_k^2}-\bar\sigma_k(u+b_k)^{-2}\;,
    \]
  giving
    \begin{align*}
      -\xi_k(x)\int_{b_k}^xq_k(t)\,dt&=\xi_k(x)\biggl(s^2\biggl(\frac{1}{3}(x-b_k)^3+b_k(x-b_k)^2\biggr)\\
        &\phantom{=\text{}}\text{}-\frac{\bar\sigma_k}{b_k^2}(x-b_k)-\bar\sigma_k(x^{-1}-b_k^{-1})\biggr)\\
        &\ge\xi_k(x)\biggl(s^2b_k(x-b_k)^2-\frac{|\bar\sigma_k|}{b_k^2}(x-b_k)-|\bar\sigma_k|\,b_k^{-1}\biggr)\\
        &\ge\xi_k(x)\biggl(\biggl(s^2b_k-\frac{|\bar\sigma_k|}{b_k^2(b_{k+1}-b_k)}\biggr)(x-b_k)^2-|\bar\sigma_k|\,b_k^{-1}\biggr)
    \end{align*}
  for $x\ge b_{k+1}$. By~\eqref{b_k+1 -b_k in O(k^-1/2)}, it follows that there is some $F_1>0$ such that
    \begin{equation}\label{F_1}
      -\int_{b_k}^xq_k(t)\xi_k(t)\,dt\ge F_1\xi_k(x)k^{1/2}(x-b_k)^2
    \end{equation}
  for all $k$ and $x\ge b_{k+1}$. Now the result follows from~\eqref{int_x_0^b_k q_k(t) xi_k(t) dt>-int_b_k^x q_k(t) xi_k(t) dt}--\eqref{F_1}.
\end{proof}

\begin{lem}\label{l:G}
  For each $\epsilon>0$, there is some $G>0$ such that, for all $k\in\N$,
    \[
      \max_{|x-x_{k,1}|\le\epsilon k^{-1/6}}\sum_{\ell=0}^{k-1}\xi_\ell^2(x)\le Gk^{1/6}\;.
    \]
\end{lem}

\begin{proof}
  Take any $x\in\R$ such that $|x-x_{k,1}|\le\epsilon k^{-1/6}$. By Corollary~\ref{c:distribution of the zeros},
    \begin{equation}\label{|x-b_k| le (epsilon+C_1'')k^-1/6}
      |x-b_k|\le |x-x_{k,1}|+|x_{k,1}-b_k|\le(\epsilon+C_1'')k^{-1/6}
    \end{equation}
  for $k\ge C_0''$. In particular, $b_k<x$ if $k$ is large enough. With this assumption,  let $\ell_0,\ell_1,\ell_2\in\N$ satisfying $0<\ell_0<\ell_1<\ell_2-1$, where $\ell_0$ and $\ell_1$ will be determined later, and $\ell_2$ is the maximum of the naturals $\ell<k$ with $b_{\ell'}\le x$ for all $\ell'\le\ell$. Let $f_\pm(t)=\sqrt{2t+1+2\sigma\pm1}$ for $t\ge1$. We have
    \[
      f_\pm(\ell)-\sqrt{s}b_\ell=2\,\frac{\pm(2\ell+1+2\sigma)+1+\bar\sigma_\ell}{\bigl(2\ell+1+2\sigma\pm2-\sqrt{(2\ell+1+2\sigma)^2-4\bar\sigma_\ell}\bigr)\bigl(f_\pm(\ell)+\sqrt{s}b_\ell\bigr)}
    \]
  for $\ell\in\Z_+$. So, assuming that $k$ is large enough, we can fix $\ell_0$, independently of $k$ and $x$, so that $f_-(\ell)<\sqrt{s}b_\ell<f_+(\ell)$ for all $\ell\ge\ell_0$. We have $f_+(\ell_1)<f_-(\ell_2)$ because $\ell_1<\ell_2-1$. Moreover observe that
    $$
      f_+'(t)=(2(t+1+\sigma))^{-1/2}>0\;,\quad
      f_+''(t)=-(2(t+1+\sigma))^{-3/2}<0
    $$
  for all $t\ge1$. Then, by Lemma~\ref{l:F},
    \begin{multline*}
      \sum_{\ell=\ell_0}^{\ell_1-1}\xi_\ell^2(x)\le\sum_{\ell=\ell_0}^{\ell_1-1}\frac{F^2\ell^{-5/6}}{(x-b_\ell)^4}\le F^2\sum_{\ell=\ell_0}^{\ell_1-1}\frac{\ell^{-5/6}}{(b_{\ell_2}-b_\ell)^4}\\
      \le F^2\sqrt{s}\sum_{\ell=\ell_0}^{\ell_1-1}\frac{\ell^{-5/6}}{(f_-(\ell_2)-f_+(\ell))^4}\le F^2\sqrt{s}\int_{\ell_0}^{\ell_1}\frac{t^{-5/6}\,dt}{(f_-(\ell_2)-f_+(t))^4}\;.
    \end{multline*}
  After integrating by parts four times, we get
    \begin{multline*}
      \int_{\ell_0}^{\ell_1}\frac{t^{-5/6}\,dt}{(f_-(\ell_2)-f_+(t))^4}\le\frac{\ell_1^{-5/6}{f_+'}^{-1}(\ell_1)}{3(f_-(\ell_2)-f_+(\ell_1))^3}+\frac{5\ell_1^{-11/6}{f_+'}^{-2}(\ell_1)}{36\,(f_-(\ell_2)-f_+(\ell_1))^2}\\
      \text{}+\frac{55\,\ell_1^{-17/6}{f_+'}^{-3}(\ell_1)}{216\,(f_-(\ell_2)-f_+(\ell_1))}+\frac{935}{1296}\ell_1^{-23/6}{f_+'}^{-4}(\ell_1)\ln(f_-(\ell_2))\\
      \text{}+\frac{21505}{7776}\ln(f_-(\ell_2))\int_{\ell_0}^{\ell_1}t^{-29/6}{f_+'}^{-4}(t)\,dt\;.
    \end{multline*}
  Therefore, since $f_+'(t)\in O(t^{-1/2})$ as $t\to\infty$, there exists some $G_1>0$, independent of $k$ and $x$, such that
    \begin{multline*}
      \sum_{\ell=\ell_0}^{\ell_1-1}\xi_\ell^2(x)\le G_1\biggl(\frac{\ell_1^{-1/3}}{(f_-(\ell_2)-f_+(\ell_1))^3}+\frac{\ell_1^{-5/6}}{(f_-(\ell_2)-f_+(\ell_1))^2}\\
      \text{}+\frac{\ell_1^{-4/3}}{f_-(\ell_2)-f_+(\ell_1)}+\ell_1^{-11/6}\ln(f_-(\ell_2))+\ln(f_-(\ell_2))\biggr)\;.
    \end{multline*}
  We have $\ell_1^{-11/6}\ln(f_-(\ell_2))+\ln(f_-(\ell_2))\le\ell_2^{1/6}$ for $k$ large enough. Then $\sum_{\ell=1}^{\ell_0-1}\xi_\ell^2(x)$ has an upper bound of the type of the statement if $\ell_1$ satisfies
    \begin{equation}\label{max le ell_2^1/6}
      \max\biggl\{\frac{\ell_1^{-1/3}}{(f_-(\ell_2)-f_+(\ell_1))^3},\frac{\ell_1^{-5/6}}{(f_-(\ell_2)-f_+(\ell_1))^2},\frac{\ell_1^{-4/3}}{f_-(\ell_2)-f_+(\ell_1)}\biggr\}\le\ell_2^{1/6}\;.
    \end{equation}
    
  On the other hand, according to Theorem~\ref{t:upper estimates of xi k}-(ii),(iii),
    \[
      \sum_{\ell_1}^{\ell_2}\xi_\ell^2(x)\le\bar C\sum_{\ell_1}^{\ell_2}\ell^{-1/6}\\
      \le\bar C\int_{\ell_1}^{\ell_2}y^{-1/6}\,dy=\frac{6\bar C}{5}\,(\ell_2^{5/6}-\ell_1^{5/6})\;,
    \]
  where $\bar C=\max\{C',C''\}$. Then $\sum_{\ell=\ell_1}^{\ell_2}\xi_\ell^2(x)$ has an upper bound of the type of the statement if $\ell_2^{5/6}-\ell_1^{5/6}\le G_2\ell_2^{1/6}$ for some $G_2>0$, independent of $k$ and $x$, which is equivalent to
    \begin{equation}\label{ell_1 ge ...}
      \ell_1\ge\ell_2\bigl(1-G_2\ell_2^{-2/3}\bigr)^{6/5}\;.
    \end{equation}
  Thus we must check the compatibility of~\eqref{max le ell_2^1/6} with~\eqref{ell_1 ge ...} for some $\ell_1$ and $G_2$. By~\eqref{ell_1 ge ...} and since, for each $G_2,\delta>0$, we have $G_2\ell_2^{-2/3}\le\ell_2^{-\frac{2}{3}+\delta}$ for $k$ large enough, we can replace~\eqref{max le ell_2^1/6} with
    \begin{multline*}
      \max\biggl\{\frac{\ell_2^{-1/3}\bigl(1-\ell_2^{-\frac{2}{3}+\delta}\bigr)^{-2/5}}{(f_-(\ell_2)-f_+(\ell_1))^3},\frac{\ell_2^{-5/6}\bigl(1-\ell_2^{-\frac{2}{3}+\delta}\bigr)^{-1}}{(f_-(\ell_2)-f_+(\ell_1))^2},\\
      \frac{\ell_2^{-4/3}\bigl(1-\ell_2^{-\frac{2}{3}+\delta}\bigr)^{-8/5}}{f_-(\ell_2)-f_+(\ell_1)}\biggr\}\le\ell_2^{1/6}
    \end{multline*}
  for some $\delta>0$, which is equivalent to
    \[
      \ell_1\le\frac{1}{2}\biggl(\sqrt{2(\ell_2+\sigma)}-\ell_2^a\bigl(1-\ell_2^{-\frac{2}{3}+\delta}\bigr)^b\biggr)^2-1-\sigma
    \]
  for
    \[
      (a,b)\in\{(-1/6,-2/15),(-1/2,-1/2),(-3/2,-8/5)\}\;.
    \]
  Thus the compatibility of~\eqref{max le ell_2^1/6} with~\eqref{ell_1 ge ...} holds if there is some $G_2,\delta>0$ such that
    \[
    \ell_2\bigl(1-G_2\ell_2^{-2/3}\bigr)^{6/5}\le\frac{1}{2}\biggl(\sqrt{2(\ell_2+\sigma)}-\ell_2^a\bigl(1-\ell_2^{-\frac{2}{3}+\delta}\bigr)^b\biggr)^2-2-\sigma\;,
    \]
  which is  equivalent to
    \[
    G_2\ge\ell_2^{2/3}\biggl(1-\biggl(\frac{1}{2}\biggl(\sqrt{2(1+\sigma\ell_2^{-1})}-\ell_2^{a-\frac{1}{2}}\bigl(1-\ell_2^{-\frac{2}{3}+\delta}\bigr)^b\biggr)^2-(2+\sigma)\ell_2^{-1}\biggr)^{5/6}\biggr)\;.
    \]
  There is some $G_2>0$ satisfying this condition because the l'H\^ospital rule shows that, for $\delta$ small enough, each function
    \[
    t^{2/3}\biggl(1-\biggl(\frac{1}{2}\biggl(\sqrt{2(1+\sigma t^{-1})}-t^{a-\frac{1}{2}}\bigl(1-t^{-\frac{2}{3}+\delta}\bigr)^b\biggr)^2-(2+\sigma)t^{-1}\biggr)^{5/6}\biggr)
    \]
  is convergent in $\R$ as $t\to\infty$.
  
  Now, if $\ell_2<k-1$, let $\ell_3$ denote the minimum integer $\ell<k$ such that $b_{\ell'}>x$ for all $\ell'\ge\ell$. Also, let $\bar\sigma_{\text{\rm min/max}}$ denote the minimum/maximum values of $\bar\sigma_\ell$ for $\ell\in\N$. Then
    \begin{multline*}
      \sqrt{\frac{2(\ell_3-1)+1+2\sigma+\sqrt{(2(\ell_3-1)+1+2\sigma)^2+4\bar\sigma_{\text{\rm min}}}}{2s}}\le x\\
      <\sqrt{\frac{2(\ell_2+1)+1+2\sigma+\sqrt{(2(\ell_2+1)+1+2\sigma)^2+4\bar\sigma_{\text{\rm max}}}}{2s}}\;,
    \end{multline*}
  obtaining
    \begin{multline*}
      2(\ell_3-\ell_2)-4\\
      <\sqrt{(2(\ell_2+1)+1+2\sigma)^2+4\bar\sigma_{\text{\rm max}}}-\sqrt{(2(\ell_3-1)+1+2\sigma)^2+4\bar\sigma_{\text{\rm min}}}\;.
    \end{multline*}
  If $\ell_3>\ell_2+1$, it follows that
    \[
      (2(\ell_2+1)+1+2\sigma)^2+4\bar\sigma_{\text{\rm max}}>(2(\ell_3-1)+1+2\sigma)^2+4\bar\sigma_{\text{\rm min}}\;,
    \]
  giving
    \begin{multline*}
      2\sqrt{\bar\sigma_{\text{\rm max}}-\bar\sigma_{\text{\rm min}}}>\sqrt{(2(\ell_2+1)+1+2\sigma)^2-(2(\ell_3-1)+1+2\sigma)^2}\\
      \ge2(\ell_3-\ell_2)-4\;.
    \end{multline*}
  Therefore $\sum_{\ell=\ell_2+1}^{\ell_3}\xi^2(x)$ has an upper bound of the type of the statement by Theorem~\ref{t:upper estimates of xi k}-(ii),(iii).
  
  Let $h(t)=(2t+1+2\sigma)s-s^2x^2-\bar\sigma_{\text{\rm max}}x^{-2}$ for $t\ge0$. According to Theorem~\ref{t:upper estimates of xi k}-(i), if $\ell_3<k-1$, then
    \begin{multline*}
      \sum_{\ell=\ell_3+1}^{k-1}\xi_\ell^2(x)\le C\sum_{\ell=\ell_3+1}^{k-1}\frac{1}{\sqrt{q_\ell(x)}}\le C\sum_{\ell=\ell_3+1}^{k-1}\frac{1}{\sqrt{h(\ell)}}\le C\int_{\ell_3}^{k-1}\frac{dt}{\sqrt{h(t)}}\\
      =\frac{C}{2s}\bigl(\sqrt{h(k-1)}-\sqrt{h(\ell_3)}\bigr)\le\frac{C}{2s}\sqrt{2(k-1-\ell_3)}\;.
    \end{multline*}
  Hence $\sum_{\ell=\ell_3+1}^{k-1}\xi_\ell^2(x)$ also has an upper bound like in the statement because, by~\eqref{b_k-b_ell ge C(k-ell)k^-1/2},~\eqref{b_k+1 -b_k in O(k^-1/2)} and~\eqref{|x-b_k| le (epsilon+C_1'')k^-1/6}, there is some $G_3,G_4>0$ such that
     \[
       G_3(k-1-\ell_3)k^{-1/2}\le b_{k-1}-b_{\ell_3}\le b_{k-1}-x\le G_4k^{-1/6} \;.\qed
     \]
\renewcommand{\qed}{}
\end{proof}

\begin{proof}[Proof of Theorem~\ref{t:lower estimates of max xi k^2}]
  By~\eqref{Gauss-Jacobi},
    \[
      1=\int_{-\infty}^\infty\left(\frac{p_k(x)}{x-x_{k,1}}\right)^2\frac{|x|^{2\sigma}e^{-sx^2}}{{p_k'}^2(x_{k,1})\,\lambda_{k,1}}\,dx\;.
    \]
  Thus, by~\eqref{p^2(x)} and Lemma~\ref{l:G},
    \begin{multline*}
      \int_{|x-x_{k,1}|\le\epsilon k^{-1/6}}\left(\frac{p_k(x)}{x-x_{k,1}}\right)^2\frac{|x|^{2\sigma}e^{-sx^2}}{{p_k'}^2(x_{k,1})\,\lambda_{k,1}}\,dx\\
      \le\int_{|x-x_{k,1}|\le\epsilon k^{-1/6}}\sum_{\ell=0}^{k-1}\xi_\ell^2(x)\,dx\le2\epsilon k^{-1/6}\max_{|x-x_{k,1}|\le\epsilon k^{-1/6}}\sum_{\ell=0}^{k-1}\xi_\ell^2(x)\le2\epsilon G
    \end{multline*}
  for any $\epsilon>0$. It follows that
    \begin{equation}\label{... ge frac 1 2}
      \int_{|x-x_{k,1}|\ge\epsilon k^{-1/6}}\left(\frac{p_k(x)}{x-x_{k,1}}\right)^2\frac{|x|^{2\sigma}e^{-sx^2}}{{p_k'}^2(x_{k,1})\,\lambda_{k,1}}\,dx\ge\frac{1}{2}
    \end{equation}
  when $\epsilon\le\frac{1}{4G}$, which implies part~(i) by Lemma~\ref{l:p k' 2(x k,i) lambda k,i}.
  
  When $k$ is even and $\sigma<0$, either $0<x_{k,k/2}<a_k$ or $|x_{k,k/2}-a_k|\le C''_1k^{-1/6}$ for $k$ large enough according to Corollary~\ref{c:distribution of the zeros}. Moreover $|x_{k,1}-b_k|\le C''_1k^{-1/6}$ for $k$ large enough by Corollary~\ref{c:distribution of the zeros} as well. So, by~\eqref{a_k in O(k^-1/2)} and~\eqref{s(b k-a k) 2}, there are some $C_0,C_1>0$, independent of $k$, such that
    \begin{gather*}
      x_{k,k/2}\le a_k+C''_1k^{-1/6}\le C_0k^{-1/2}\;,\\
      x_{k,1}-x_{k,k/2}\ge b_k-a_k-2C''_1k^{-1/6}=\sqrt{\frac{c_{\text{\rm max}}}{s}}-2C''_1k^{-1/6}\ge C_1k^{1/2}
    \end{gather*}
  On the other hand, by~\eqref{p k(0)}, there is some $C_2>0$, independent of $k$, such that $\xi_k^2(x)\le C_2\,|x|^{2\sigma}$ for $|x|\le x_{k,k/2}$. Therefore
    \begin{multline*}
      \int_{|x|\le x_{k,k/2}}\frac{\xi_k^2(x)\,dx}{(x-x_{k,1})^2}\le\frac{C_2}{(x_{k,1}-x_{k,k/2})^2}\int_{|x|\le x_{k,k/2}}|x|^{2\sigma}\,dx\\
      =\frac{2C_2x_{k,k/2}^{2\sigma+1}}{(2\sigma+1)(x_{k,1}-x_{k,k/2})^2}\le\frac{2C_2C_0^{2\sigma+1}}{(2\sigma+1)C_1^2}\,k^{-\frac{2\sigma+3}{2}}<\frac{2C_2C_0^{2\sigma+1}}{(2\sigma+1)C_1^2}\,k^{-1}\;.
    \end{multline*}
  This inequality and~\eqref{... ge frac 1 2} imply part~(ii).
\end{proof}



\providecommand{\bysame}{\leavevmode\hbox to3em{\hrulefill}\thinspace}
\providecommand{\MR}{\relax\ifhmode\unskip\space\fi MR }
\providecommand{\MRhref}[2]{%
  \href{http://www.ams.org/mathscinet-getitem?mr=#1}{#2}
}
\providecommand{\href}[2]{#2}

\end{document}